\documentclass{article}

\usepackage{authblk}
\usepackage{amsmath}
\usepackage{amssymb}
\usepackage{amsthm}
\usepackage{booktabs}
\usepackage[margin=1in]{geometry}
\usepackage[hidelinks]{hyperref}

\newcommand{\Rplayer}{\mathfrak{R}}
\newcommand{\Dplayer}{\mathfrak{D}}

\newtheorem{lemma}{Lemma}
\newtheorem{theorem}{Theorem}
\newtheorem{corollary}{Corollary}

\title{A Combinatorial Generalization of a Random-Player Game}
\author{Yehonatan Fridman}
\affil{Ben-Gurion University, NRCN, Israel \thanks{fridyeh@post.bgu.ac.il}}
\date{}

\begin{document}

\maketitle

\begin{abstract}
In a previous note, a two-player game between a random player and a
deterministic player was introduced, and it was shown analytically that the
winning probability of the deterministic player is the derangement probability
$d_n/n!$. The natural question left open was to explain this coincidence
combinatorially. This paper gives such an explanation and extends it to a
larger family of games. In the generalized game, the deterministic player
removes $q$ elements per turn, or all remaining elements if fewer than $q$
remain. We couple the game exactly to the cycle decomposition of a uniformly
random permutation. Under this coupling, the random player wins precisely when
the first cycle of length at most $q$, read in canonical cycle order, is a
fixed point. The case $q=1$ recovers derangements, while the general case is
governed by the first short cycle of the permutation.
\end{abstract}

\subsection*{Keywords}
Random games; random permutations; cycle decomposition;\\
derangements; short cycles

\section{Introduction}

The random-player game studied in~\cite{fridman2024random} starts with a pile
of $n$ elements. On each turn, the random player $\Rplayer$ removes a uniformly
chosen number of elements from the current pile, while the deterministic player
$\Dplayer$ removes one element. It was shown there, by an analytic recurrence
calculation, that the probability that $\Dplayer$ wins is
\[
    \frac{d_n}{n!},
\]
where $d_n$ is the number of derangements of an $n$-element set. Equivalently,
the analytic solution gave
\[
    \mathbb{P}(\Dplayer\text{ wins from }n\text{ elements})
    =
    \sum_{k=0}^{n}\frac{(-1)^k}{k!},
    \qquad
    \lim_{n\to\infty}
    \mathbb{P}(\Dplayer\text{ wins from }n\text{ elements})
    =e^{-1}.
\]
Thus the game has the same winning probability as the event that a uniformly
random permutation has no fixed points.

The analytic calculation proves the identity, but it does not explain why a
game on a pile should know about derangements. The main purpose of this note is
to give a direct combinatorial explanation. We also show that the same idea
naturally extends to a more general game in which $\Dplayer$ removes $q$
elements per turn. In this generalized game, derangements are replaced by the
first cycle of length at most $q$ in a uniformly random permutation.

The needed facts about random permutation cycles are classical: cycle
decompositions, derangements, short cycles, and Poisson limits for fixed cycle
lengths are standard; see, for example, Feller~\cite{feller} and Ford's toolkit
on cycle types of random permutations~\cite{ford2022cycle}. The contribution
here is the exact coupling between the game and the canonical cycle order of a
uniformly random permutation.

\section{The generalized game}

Fix a positive integer $q$. The game starts with a pile of $n$ elements. The
players move alternately, and $\Rplayer$ moves first. If there are $m$ elements
at the start of $\Rplayer$'s turn, then $\Rplayer$ removes a uniformly chosen
number $k\in\{1,\ldots,m\}$. If this empties the pile, $\Rplayer$ wins. If not,
then $\Dplayer$ removes $\min(q,m')$ elements, where $m'$ is the number of
elements remaining after $\Rplayer$'s move. If this empties the pile,
$\Dplayer$ wins. Otherwise the game continues.

Let $R_n^{(q)}$ and $D_n^{(q)}$ be the probabilities that $\Rplayer$ and
$\Dplayer$ win, respectively, when the initial pile has $n$ elements. Thus
$R_n^{(q)}+D_n^{(q)}=1$.

The case $q=1$ is the original game: $\Dplayer$ removes one element on each
turn. The purpose of this note is to explain this case, and all larger values
of $q$, directly through the cycle structure of random permutations.

\section{Revealing cycles}

Let $\pi$ be a uniformly random permutation of a finite set $S$. We write its
cycles in \emph{canonical order}: first the cycle containing the smallest
element of $S$, then the cycle containing the smallest element not yet written,
and so on.

\begin{lemma}
\label{lem:uniform-cycle-length}
Suppose $|S|=m$. In a uniformly random permutation of $S$, the length $L$ of
the cycle containing the smallest element of $S$ is uniformly distributed on
$\{1,\ldots,m\}$. After this cycle is revealed and removed, the permutation
induced on the remaining elements is again uniform.
\end{lemma}

\begin{proof}
Let $a$ be the smallest element of $S$. For a fixed length $\ell$, the number
of permutations in which the cycle containing $a$ has length $\ell$ is
\[
    \binom{m-1}{\ell-1}(\ell-1)!(m-\ell)!=(m-1)!.
\]
This number is independent of $\ell$, so $L$ is uniform on
$\{1,\ldots,m\}$.

For any fixed revealed cycle, the remaining elements may be permuted
arbitrarily. Hence the induced permutation on the remaining set is uniform.
\end{proof}

\section{The coupling}

We now couple the $q$-game with a uniformly random permutation of
$\{1,\ldots,n\}$. Reveal the cycles of the permutation in canonical order.
Assume that $m$ elements have not yet been revealed, and let $L$ be the length
of the next cycle.

Define $\Rplayer$'s move $K$ from $L$ as follows:
\[
K=
\begin{cases}
m, & L=1,\\
m-L+1, & 2\le L\le \min(q,m),\\
L-q, & q<L\le m.
\end{cases}
\]
This map sends the possible values of $L$ bijectively to
$\{1,\ldots,m\}$. If $m\le q$, then $L=1$ gives $K=m$, while
$L=2,\ldots,m$ gives $K=m-1,\ldots,1$. If $m>q$, then
$L=q+1,\ldots,m$ gives $K=1,\ldots,m-q$, the values
$L=2,\ldots,q$ give $K=m-1,\ldots,m-q+1$, and $L=1$ gives
$K=m$. By Lemma~\ref{lem:uniform-cycle-length}, $L$ is uniform on
$\{1,\ldots,m\}$, and therefore $K$ is uniform on $\{1,\ldots,m\}$. Thus the
coupled process gives exactly the same random move that $\Rplayer$ makes in
the game.

The outcome is now determined by the length of the next cycle.

\begin{itemize}
\item If $L=1$, then $K=m$, so $\Rplayer$ removes the whole pile and wins.
\item If $2\le L\le q$, then $\Rplayer$ leaves $L-1$ elements. Since
      $L-1\le q$, $\Dplayer$ removes them all and wins.
\item If $L>q$, then $\Rplayer$ removes $L-q$ elements and $\Dplayer$ removes
      the next $q$ elements. Together they remove $L$ elements, exactly the
      length of the revealed cycle. If elements remain, the game continues
      with the next cycle; if no elements remain, $\Dplayer$ wins.
\end{itemize}

Thus the game is equivalent to reading the cycles of a uniformly random
permutation until a cycle of length at most $q$ is encountered.

\begin{theorem}
\label{thm:main-coupling}
For every $n$ and every positive integer $q$, the following description gives
the exact law of the $q$-game. Let $\pi$ be a uniformly random permutation of
$\{1,\ldots,n\}$, and read its cycles in canonical order.

The player $\Rplayer$ wins if and only if the first cycle of length at most
$q$ is a fixed point. The player $\Dplayer$ wins if and only if either the
first cycle of length at most $q$ has length in $\{2,\ldots,q\}$, or no cycle
of length at most $q$ exists.
\end{theorem}

\begin{proof}
The preceding construction gives a move-by-move coupling between the game and
the cycle-revealing process. Cycles of length greater than $q$ are consumed by
a full round of play, after which the same construction restarts on the
remaining elements. The first cycle of length at most $q$ ends the game:
length $1$ gives a win for $\Rplayer$, while lengths $2,\ldots,q$ give a win
for $\Dplayer$. If all cycles have length greater than $q$, the last cycle is
also consumed by a full round, and $\Dplayer$ empties the pile.
\end{proof}

\section{Examples}

Let $q=3$. Consider the permutation
\[
    (1\,5\,8\,2)(3\,7)(4)(6\,9\,10).
\]
The cycles are already written in canonical order. The first cycle has length
$4$, so one full round of the game consumes that cycle and the game continues.
The next cycle has length $2$, which is at most $q$ and is not a fixed point.
Therefore $\Dplayer$ wins. The fixed point $(4)$ is irrelevant, because the
game has already ended.

On the other hand, for
\[
    (1\,5\,8\,2)(3)(4\,7)(6\,9\,10),
\]
the first cycle of length at most $3$ is the fixed point $(3)$, so
$\Rplayer$ wins.

Finally, if
\[
    (1\,4\,7\,2)(3\,6\,9\,5)(8\,10\,11\,12)
\]
and $q=3$, then there are no cycles of length at most $3$. In the coupled
game each cycle is consumed by a full round, and $\Dplayer$ wins on the last
round.

\section{The derangement case}

When $q=1$, the only cycles of length at most $q$ are fixed points. By
Theorem~\ref{thm:main-coupling}, $\Rplayer$ wins if and only if the
permutation has at least one fixed point. Therefore $\Dplayer$ wins if and
only if the permutation has no fixed points.

\begin{corollary}
For the original game, where $\Dplayer$ removes one element on each turn,
\[
    D_n^{(1)}=\frac{d_n}{n!},
\]
where $d_n$ is the number of derangements of $\{1,\ldots,n\}$.
\end{corollary}

This gives a direct combinatorial explanation for the appearance of
derangement numbers: $\Dplayer$ wins exactly on fixed-point-free
permutations.

\section{An exact finite expression}

Let $C_j(\pi)$ denote the number of cycles of length $j$ in a permutation
$\pi$, and set
\[
    M_q(\pi)=\sum_{j=1}^q jC_j(\pi),
\]
the number of elements lying in cycles of length at most $q$.

Theorem~\ref{thm:main-coupling} says that $\Rplayer$ wins exactly when the
smallest element among these $M_q(\pi)$ elements lies in a fixed point. If
$M_q(\pi)=0$, then $\Rplayer$ does not win. Conditional on the numbers
$C_1,\ldots,C_q$, the labels on the elements lying in short cycles are
exchangeable. There are $C_1$ elements in fixed points among the $M_q$
short-cycle elements. Hence the exact finite probability is
\[
    R_n^{(q)}
    =
    \mathbb{E}\left[
        \frac{C_1}{M_q}\mathbf{1}_{\{M_q>0\}}
    \right],
    \qquad
    D_n^{(q)}
    =
    1-
    \mathbb{E}\left[
        \frac{C_1}{M_q}\mathbf{1}_{\{M_q>0\}}
    \right].
\]

This formula is often the cleanest finite form: it says that one first chooses
a uniformly random permutation, then looks only at the elements lying in cycles
of lengths $1,\ldots,q$. The random player wins exactly when the smallest of
these elements lies in a fixed point.

\section{Limiting probabilities and computation}

For fixed $q$, the short cycle counts of a uniformly random permutation have
a standard Poisson limit:
\[
    (C_1,\ldots,C_q)
    \xrightarrow[n\to\infty]{d}
    (Z_1,\ldots,Z_q),
\]
where the random variables $Z_1,\ldots,Z_q$ are independent and
$Z_j\sim \operatorname{Poisson}(1/j)$.

This limit has a short combinatorial proof. Here
$(x)_a=x(x-1)\cdots(x-a+1)$ denotes a falling factorial. For fixed
nonnegative integers $a_1,\ldots,a_q$, the factorial moment
\[
    \mathbb{E}\left[\prod_{j=1}^q (C_j)_{a_j}\right]
\]
counts ordered choices of $a_j$ cycles of length $j$, for each $j$. Choosing
these disjoint cycles and then permuting all remaining elements gives exactly
\[
    \prod_{j=1}^q \frac{1}{j^{a_j}},
\]
whenever $n\ge \sum_{j=1}^q ja_j$. These are the factorial moments of
independent Poisson variables with means $1,1/2,\ldots,1/q$.

Thus, with
\[
    W_q=Z_1+2Z_2+\cdots+qZ_q,
\]
we obtain the limiting form
\[
    \lim_{n\to\infty} R_n^{(q)}
    =
    \mathbb{E}\left[
        \frac{Z_1}{W_q}\mathbf{1}_{\{W_q>0\}}
    \right],
    \qquad
    \lim_{n\to\infty} D_n^{(q)}
    =
    1-
    \mathbb{E}\left[
        \frac{Z_1}{W_q}\mathbf{1}_{\{W_q>0\}}
    \right].
\]

Equivalently, the limiting event is described as follows. Among all elements
that lie in cycles of lengths $1,\ldots,q$, the smallest such element decides
the game. If it lies in a fixed point, $\Rplayer$ wins; if it lies in a cycle
of length $2,\ldots,q$, $\Dplayer$ wins. If no such element exists,
$\Dplayer$ wins.

This expectation can be converted into a one-dimensional formula. Let
\[
    H_q=\sum_{j=1}^{q}\frac{1}{j},
    \qquad
    P_q(x)=\sum_{j=1}^{q}\frac{x^j}{j}.
\]
Here $H_q$ is the $q$-th harmonic number. Since
\[
    \frac{1}{W_q}=\int_0^1 x^{W_q-1}\,dx
    \qquad (W_q>0),
\]
and the variables $Z_1,\ldots,Z_q$ are independent Poisson variables, we get
\[
\begin{aligned}
    \rho_q
    :=
    \lim_{n\to\infty} R_n^{(q)}
    &=
    \mathbb{E}\left[
        \frac{Z_1}{W_q}\mathbf{1}_{\{W_q>0\}}
    \right]                                                  \\
    &=
    \int_0^1 \mathbb{E}\left[Z_1x^{W_q-1}\right]\,dx          \\
    &=
    e^{-H_q}\int_0^1 e^{P_q(x)}\,dx .
\end{aligned}
\]
Consequently,
\[
    \lim_{n\to\infty} D_n^{(q)}=1-\rho_q.
\]

The values in Table~\ref{tab:limits} were computed numerically from this
one-dimensional integral.

\begin{table}[h]
\centering
\begin{tabular}{ccc}
\toprule
$q$ & $\lim_{n\to\infty}R_n^{(q)}$ & $\lim_{n\to\infty}D_n^{(q)}$ \\
\midrule
1 & 0.632120559 & 0.367879441 \\
2 & 0.478268141 & 0.521731859 \\
3 & 0.390581907 & 0.609418093 \\
4 & 0.332928681 & 0.667071319 \\
5 & 0.291712632 & 0.708287368 \\
10 & 0.186039907 & 0.813960093 \\
20 & 0.113235137 & 0.886764863 \\
50 & 0.055892559 & 0.944107441 \\
100 & 0.031913064 & 0.968086936 \\
\bottomrule
\end{tabular}
\caption{Selected limiting winning probabilities.}
\label{tab:limits}
\end{table}

The table suggests that, after taking $n\to\infty$, the random player's
winning probability decreases to zero as $q$ grows. This is indeed the case.

\begin{corollary}
\[
    \lim_{q\to\infty}\rho_q=0,
    \qquad
    \lim_{q\to\infty}\lim_{n\to\infty}D_n^{(q)}=1.
\]
\end{corollary}

\begin{proof}
For $0\le x<1$,
\[
    P_q(x)\le \sum_{j=1}^{\infty}\frac{x^j}{j}=-\log(1-x).
\]
For $q\ge2$, split the integral defining $\rho_q$ at $1-1/q$. On the first
part,
\[
    e^{P_q(x)}\le \frac{1}{1-x},
\]
and on the second part, $P_q(x)\le H_q$. Therefore
\[
\begin{aligned}
    \rho_q
    &\le
    e^{-H_q}\int_0^{1-1/q}\frac{dx}{1-x}
    +
    e^{-H_q}\int_{1-1/q}^{1}e^{H_q}\,dx \\
    &=
    e^{-H_q}\log q+\frac{1}{q}.
\end{aligned}
\]
Since $H_q\ge \log(q+1)$, we have
\[
    \rho_q\le \frac{\log q}{q+1}+\frac{1}{q}\xrightarrow[q\to\infty]{}0.
\]
\end{proof}

\section{Conclusion}

The derangement identity proved analytically in~\cite{fridman2024random} is
the first case of a more general cycle phenomenon. When $\Dplayer$ removes
$q$ elements per turn, the game is governed by the first cycle of length at
most $q$ in the canonical cycle order of a uniformly random permutation. The
case $q=1$ gives derangements because the only relevant short cycles are fixed
points. For larger $q$, the same coupling leads to limiting probabilities
determined by the independent Poisson limits of short cycle counts, and these
limits can be computed from the explicit formula
\[
    \rho_q=e^{-H_q}\int_0^1 e^{P_q(x)}\,dx,
    \qquad
    \lim_{n\to\infty}D_n^{(q)}=1-\rho_q .
\]
Thus the original game is not an isolated derangement coincidence, but the
$q=1$ member of a natural family controlled by short cycles in random
permutations.

\section*{Remark on the AI-discovered argument}

A notable aspect of this note is its discovery process. Over the previous two
years, the author periodically asked contemporary versions of ChatGPT for a
combinatorial explanation of the derangement identity in~\cite{fridman2024random}.
Those attempts produced answers that were either incorrect or unnecessarily
complicated. In contrast, when the question was asked again using ChatGPT 5.5,
the cycle-coupling argument in this note was found quickly and in essentially
the form presented above. This is recorded because the argument is short,
correct, and nontrivial, and because it illustrates the increasing ability of
AI tools to identify meaningful combinatorial structure.

\end{document}